\documentclass{amsart}
\usepackage{amsfonts}
\usepackage{amssymb}
\usepackage{amsmath}
\usepackage[dvips]{epsfig,graphics}
\usepackage{psfrag}
\usepackage{amscd}
\newtheorem{theorem}{Theorem}[section]
\newtheorem{lemma}[theorem]{Lemma}
\newtheorem{prop}[theorem]{Proposition}
\newtheorem{question}[theorem]{Question}
\theoremstyle{definition}
\newtheorem{definition}[theorem]{Definition}

\theoremstyle{remark}
\newtheorem{remark}[theorem]{Remark}
\numberwithin{equation}{section}

\begin{document}

\title{Twisting quasi-alternating links}
\author{Abhijit Champanerkar}
\address{Department of Mathematics, College of Staten Island, City University of New York}
\email{abhijit@math.csi.cuny.edu}
\thanks{The first author is supported by NSF grant DMS-0844485.}

\author{Ilya Kofman}
\address{Department of Mathematics, College of Staten Island, City University of New York}
\email{ikofman@math.csi.cuny.edu}
\thanks{The second author is supported by NSF grant DMS-0456227 and a PSC-CUNY grant.}

\subjclass[2000]{Primary 57M25}
\keywords{Khovanov homology, knot Floer homology, pretzel link}
\date{\today}

\begin{abstract}
\noindent
Quasi-alternating links are homologically thin for both Khovanov
homology and knot Floer homology.  We show that every
quasi-alternating link gives rise to an infinite family of
quasi-alternating links obtained by replacing a crossing with an
alternating rational tangle.  Consequently, we show that many pretzel
links are quasi-alternating, and we determine the thickness of
Khovanov homology for ``most'' pretzel links with arbitrarily many
strands.
\end{abstract}
\commby{Daniel Ruberman}
\maketitle

\section{Introduction}
In \cite{CM_PO}, quasi-alternating links were shown to be homologically
thin for both Khovanov homology and knot Floer homology.  It is the
most inclusive known description of homologically thin links.  However, the
recursive definition makes it difficult to decide whether a given knot
or link is quasi-alternating.

\begin{definition}\label{qadef}(\cite{OzSz_double_covers})
The set $\mathcal{Q}$ of quasi-alternating links is the smallest set of links
satisfying the following properties:
\begin{itemize}
\item The unknot is in $\mathcal{Q}$.
\item If the link $\mathcal L$ has a diagram $L$ with a crossing $c$ such that 
\begin{enumerate}
\item both smoothings of $c$, $L_0$ and $L_{\infty}$ as in Figure \ref{skeinfig}, are in $\mathcal{Q}$;
\item $\det(L)=\det(L_0)+\det(L_{\infty})$;
\end{enumerate}
$\!$then $\mathcal L$ is in $\mathcal{Q}$.
\end{itemize}
We will say that $c$ as above is a {\em quasi-alternating crossing} of
$L$, and that $L$ is {\em quasi-alternating at} $c$.
\end{definition}

An oriented $3$--manifold $Y$ is called an {\em {\rm L}--space} if $b_1(Y)=0$
and $|H_1(Y;\,\mathbb Z)|=rk\widehat{HF}(Y)$, where $\widehat{HF}$ denotes the
Heegaard-Floer homology.  For a link $\mathcal L$ in $S^3\!$, let $\Sigma(\mathcal L)$ denote its branched double cover.
The following property is the reason for interest in quasi-alternating links:
\begin{prop}[Proposition 3.3 \cite{OzSz_double_covers}]\label{Lprop}
If $\mathcal L$ is a quasi-alternating link, then $\Sigma(\mathcal L)$ is an {\rm L}--space.
\end{prop}

In this note, we show that a quasi-alternating crossing can be
replaced by an alternating rational tangle to obtain another
quasi-alternating link (Theorem \ref{maintheorem}).  Thus, the set of
quasi-alternating links includes many non-trivial infinite families.

For pretzel links, their Khovanov homology and knot Floer homology
have been computed only for $3$--strand pretzel links (see
\cite{OSz_2004, Eftekhary, Suzuki}).  By repeatedly applying Theorem
\ref{maintheorem}, we show that many pretzel links with arbitrarily many
strands are quasi-alternating (Theorem \ref{pretzeltheorem}(1)).
Therefore, their respective homologies can be computed from the
signature and the Jones and Alexander polynomials\footnote{A closed formula
for the Kauffman bracket of arbitrary pretzel links is given in
\cite{cjp}; for the Alexander polynomial, see \cite{Hironaka} and references
therein.}. Most other pretzel links are not quasi-alternating since 
their Khovanov homology lies on exactly three diagonals 
(Theorem \ref{pretzeltheorem}(2)).  Thus,
Theorem \ref{pretzeltheorem} gives the thickness of Khovanov
homology for ``most'' pretzel links with arbitrarily many strands.

In Section \ref{10crossqa}, we apply Theorem \ref{maintheorem} to complete
the classification of quasi-alternating knots up to 10 crossings.

Finally, the Turaev genus plus two bounds the homological thickness
for both Khovanov homology and knot Floer homology \cite{Manturov, dkh,
  Lowrance}.  The Turaev genus is preserved after inserting a rational
tangle as in Theorem \ref{maintheorem}. 
We do not know of any quasi-alternating links with Turaev genus greater than one.

\begin{question}
Do there exist quasi-alternating links with arbitrary Turaev genus?
\end{question}

\section{Twisting quasi-alternating links}
Let $c$ be a quasi-alternating crossing of a link diagram $L$, as in Definition \ref{qadef}.
We consider $c$ as a $2$-tangle with marked endpoints.  Using Conway's notation for
rational tangles, let $\varepsilon(c)=\pm 1$, according to whether the
overstrand has positive or negative slope.  We will say that a
rational $2$-tangle $\tau=C(a_1,\ldots,a_m)$ {\em extends} $c$ if
$\tau$ contains $c$, and for all $i,\ \varepsilon(c)\cdot a_i \geq 1$.
In particular, $\tau$ is an alternating rational tangle.

\begin{theorem}\label{maintheorem}
  If $L$ is a quasi-alternating link diagram, let $L'$ be obtained by
  replacing any quasi-alternating crossing $c$ with an alternating
  rational tangle that extends $c$. Then $L'$ is quasi-alternating.
\end{theorem}

We start with some background needed for the proof.  For any connected
link diagram $L$, we can associate a connected graph $G(L)$, called the
Tait graph of $L$, by checkerboard coloring complementary regions of
$L$, assigning a vertex to every shaded region, an edge to every
crossing and a $\pm$ sign to every edge as follows:
$$ \includegraphics[height=1cm]{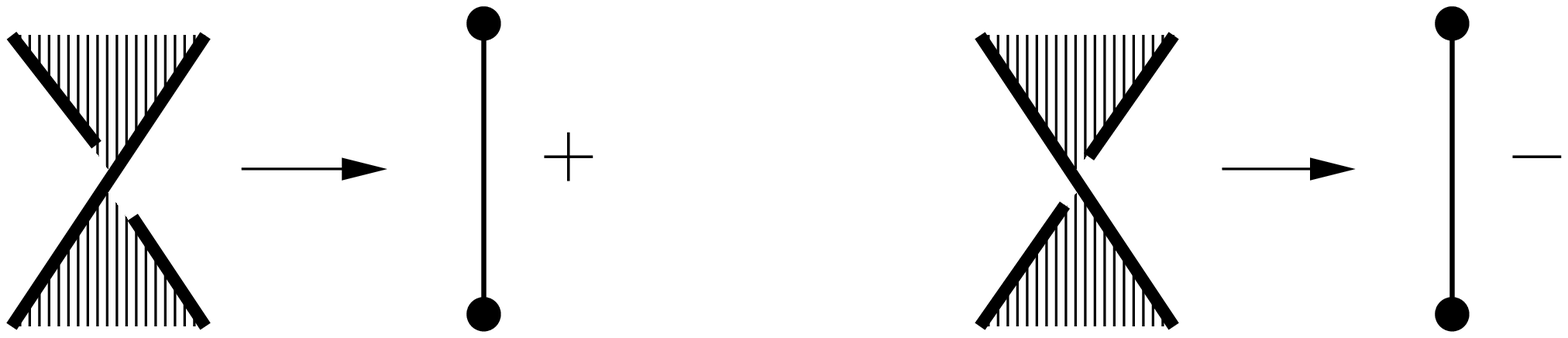}.$$
The signs are all equal if and only if $L$ is alternating.
\begin{lemma}\label{det_st}
For any spanning tree $T$ of $G(L)$, let $v(T)$ be the number
of positive edges in $T$.  Let $s_v(L)=\#\{$spanning trees $T$ of $G(L)\ |\ v(T)=v\}$.  Then
$$ \det(L) = \left|\sum_v (-1)^v\, s_v(L) \right|. $$
\end{lemma}
\begin{proof} 
Thistlethwaite \cite{thistlethwaite} gave an expansion of the
Jones polynomial $V_L(t)$ in terms of spanning trees $T$ of $G(L)$.
Using this expansion, in \cite{KHshort} for any spanning tree $T$, we defined a grading $u(T)$, such that
$$V_L(t) = (-1)^w\, t^{(3w+k)/4}\sum_{T \subset G} (-1)^{u(T)}\, t^{u(T) -v(T)}$$
where $w$ is the writhe of $L$, and $k$ is a constant that depends on $G$ 
(see the proof of Proposition 2 \cite{KHshort}).  
Thus,
$$ \det(L) = |V_L(-1)| = \left| \sum_{T \subset G} (-1)^{u(T)} (-1)^{u(T) -v(T)}\right| = \left|\sum_{v\geq 0} (-1)^v\, s_v(L) \right|.$$
\end{proof}

\begin{proof}{\em (Theorem \ref{maintheorem})\ } Let $L$ be a
quasi-alternating diagram at crossing $c$.  We may assume that
$\varepsilon(c)=1$ by rotating $L$ as needed.  For $n\in\mathbb Z$, let $L^n$ denote
the link diagram with $|n|$ additional crossings inserted at $c$,
which are vertical positive half-twists if $n>0$, and horizontal
negative half-twists if $n<0$.  The cases $n=\pm 2$ are shown in
Figure \ref{skeinfig}.  

\begin{figure}[t]
  \begin{center}
    \includegraphics[height=1cm]{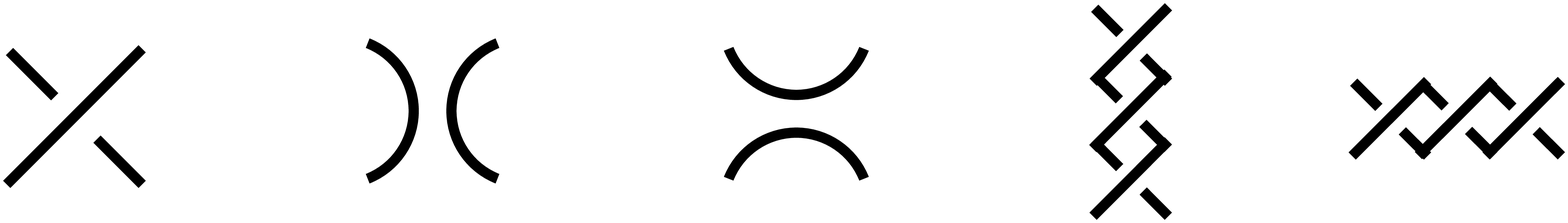}
  \end{center}
  \caption{Crossing $c$ of $L$, its smoothings $L_0$ and $L_{\infty}$,
    $L^2$ and $L^{-2}$.}
  \label{skeinfig}
\end{figure}

We first show that $L^n$ is quasi-alternating at any inserted
crossing.
Suppose $n\geq 0$.  We checkerboard color $L$ such that the edge $e$
in $G(L)$ that corresponds to $c$ is positive.  With the induced
checkerboard coloring, the Tait graph $G(L_0)$ has the edge $e$
contracted; $G(L_{\infty})$ has the edge $e$ deleted; and $G(L^n)$ has
$n$ additional vertices on $e$, dividing it into $n+1$ positive edges
$\{e_0,\ldots,e_n\}$.

For every spanning tree $T$ of $G(L)$ such that $e\in T$,
there is a unique spanning tree $T'$ of $G(L_0)$, and a unique spanning 
tree $T''$ of $G(L^n)$ such that $\{e_0,\ldots,e_n\}\subset T''$, then   
$v(T'')=v(T)+n=v(T')+1+n$. 
For every spanning tree $T$ of $G(L)$ such that $e\notin T$, 
there is a unique spanning tree $T'$ of $G(L_{\infty})$, and
there are spanning trees $T_0,\ldots,T_n$ of $G(L^n)$ such that
$e_i\notin T_i$, then for $0\leq i\leq n$, $v(T_i)=v(T)+n=v(T')+n$.
For each $v$, 
%
\begin{equation}\label{eq1}
s_v(L^n) = s_{v-n-1}(L_0) + (n+1)s_{v-n}(L_{\infty}).
\end{equation}
Therefore,
\begin{equation}\label{eq2}
\sum_v (-1)^v\, s_v(L^n) = \sum_v (-1)^v\, s_{v-n-1}(L_0) + (n+1)\sum_v (-1)^v\, 
s_{v-n}(L_{\infty}).
\end{equation}
By (\ref{eq1}) with $n=0$, $s_v(L) = s_{v-1}(L_0) + s_{v}(L_{\infty})$. Hence 
\begin{equation}\label{eq3}
\sum_v (-1)^v\, s_v(L)= \sum_v (-1)^v\, s_{v-1}(L_0) + 
\sum_v (-1)^v\, s_v(L_{\infty}).
\end{equation}
Let $x=\sum_v (-1)^v\, s_{v-1}(L_0)$ and $y=\sum_v (-1)^v\, s_v(L_{\infty})$, so $\det(L_0)=|x|$ and $\det(L_{\infty})=|y|$. 
Since $L$ is quasi-alternating at $c$, $\det(L)=\det(L_0) + \det(L_{\infty})$, that is $|x+y|=|x|+|y|$. 
Therefore, $x\cdot y\geq 0$.
It now follows from (\ref{eq2}) for $n\geq 0$,
\begin{eqnarray}
\nonumber
 \left|\sum_v (-1)^v\, s_v(L^n)\right| &=& |(-1)^nx+(n+1)(-1)^ny|=|x|+(n+1)|y| \\
%
\label{eq5}
\det(L^n) &=& \det(L_0) + (n+1)\det(L_{\infty}). 
\end{eqnarray}

Let $c$ denote any crossing in $L^n$ added to $L$ as above.
Let $L^n_0$ and $L^n_{\infty}$ denote the corresponding resolutions of 
$L^n$ at $c$. We have $L^n_{\infty}=L_{\infty}$ and $L^n_0=L^{n-1}$ as links. 
For $n\geq 1$, (\ref{eq5}) implies 
\begin{eqnarray*}
\det(L^n) &=& \det(L_0) + (n+1)\det(L_{\infty})\\
 &=& \big(\det(L_0) + n\det(L_{\infty})\big) + \det(L_{\infty})\\
 &=& \det(L^{n-1}) + \det(L^n_{\infty})\\
 &=& \det(L^n_0) + \det(L^n_{\infty}).
\end{eqnarray*}

We have that $L^0=L$ and $L^n_{\infty}=L_{\infty}$ as links, and hence are
quasi-alternating.  If $L^{n-1}=L^n_0$ is quasi-alternating, then
$L^n$ is quasi-alternating by the equations above. By induction, $L^n$
is quasi-alternating at $c$, so $L^n$ is quasi-alternating at every
inserted crossing.

Suppose $n\leq 0$.  If $L$ is quasi-alternating at $c$ then the mirror
image $\tilde{L}$ is also quasi-alternating at $c$ since $\det(\tilde{L})=\det(L)$.
Applying the argument above to $\tilde{L}$, and then reflecting proves
this case for $L$.

Since every inserted crossing is quasi-alternating in $L^n$, we can
iterate the construction above.  Let $\tau=C(a_1,\ldots,a_m)$ be an
alternating rational tangle that extends $c$.  Since $\varepsilon(c)\cdot a_i
\geq 1$ for all $i$, $L^{a_i}$ is a vertical positive twist operation,
and $L^{-a_i}$ is a horizontal negative twist operation.  We now
construct:
$$L'=  \left(\left((L^{-a_m})^{a_{m-1}}\right)^{-a_{m-2}}\cdots\right)^{(-1)^{m} (a_1-\varepsilon(c))}.$$
For example, see Figure \ref{iteration}.  The resulting link $L'$ is
quasi-alternating, and is obtained from $L$ by replacing $c$ with
$C(a_1,a_2,\ldots,a_m)$.  \end{proof}

 \begin{figure}[t]
 \begin{center}
 \includegraphics[width=5in]{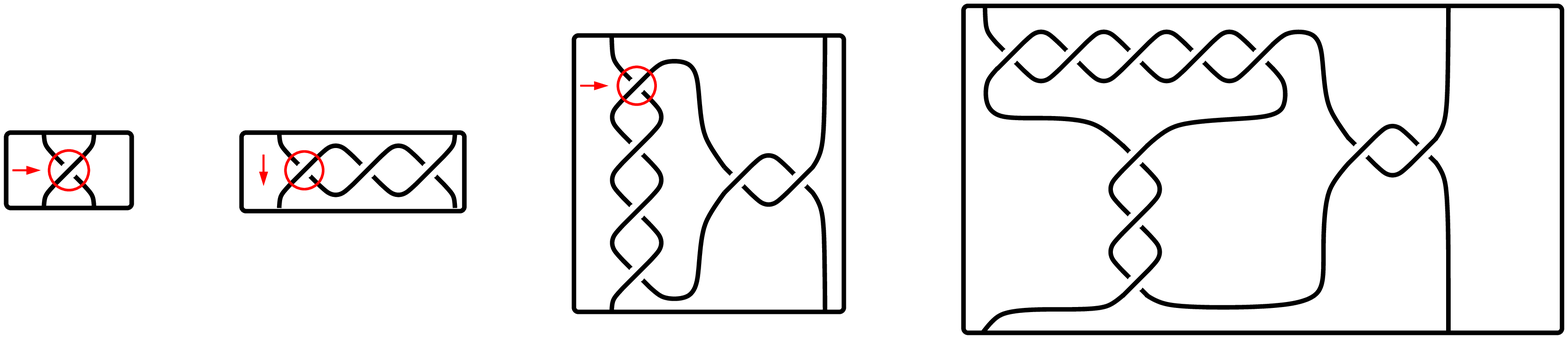}
\end{center}
 \caption{Crossing $c$ with $\varepsilon(c)=1$ is replaced by $C(5,3,2)$.  Steps shown are $L \to L^{-2} \to (L^{-2})^{3} \to
\left((L^{-2})^{3}\right)^{-4}$.}
 \label{iteration}
 \end{figure}

\ 

The following lemma will be used in the next section.

\begin{lemma} \label{connectsum}
If $K$ and $L$ are any quasi-alternating knot diagrams, then $K\# L$ is quasi-alternating. 
\end{lemma}
\begin{proof} The proof is by induction on $\det(K)$.  For quasi-alternating $K$, if $\det(K)=1$ then $K$
is the unknot, so the result follows.  Otherwise, $K$ is
quasi-alternating at a crossing $c$, so the two smoothings at $c$,
$K_0$ and $K_{\infty}$, are quasi-alternating.

Since $\det(K)=\det(K_0)+\det(K_{\infty})$, both $\det(K_0)< \det(K)$ and $\det(K_{\infty}) < \det(K)$.
By induction, both $K_0\# L$ and $K_{\infty}\# L$ are quasi-alternating.
Moreover, 
\begin{eqnarray*}
\det(K\# L) &=& \det(K)\det(L) = \left(\det(K_0)+\det(K_{\infty})\right)\det(L) \\
 &=& \det(K_0)\det(L) + \det(K_{\infty})\det(L) \\
 &=& \det(K_0\# L) + \det(K_{\infty}\# L).
\end{eqnarray*}
Therefore, $K\# L$ is quasi-alternating at $c$.
\end{proof}

\section{Pretzel links}
\begin{minipage}{2.4 in}
  \begin{center}
    \includegraphics[width=1.75 in]{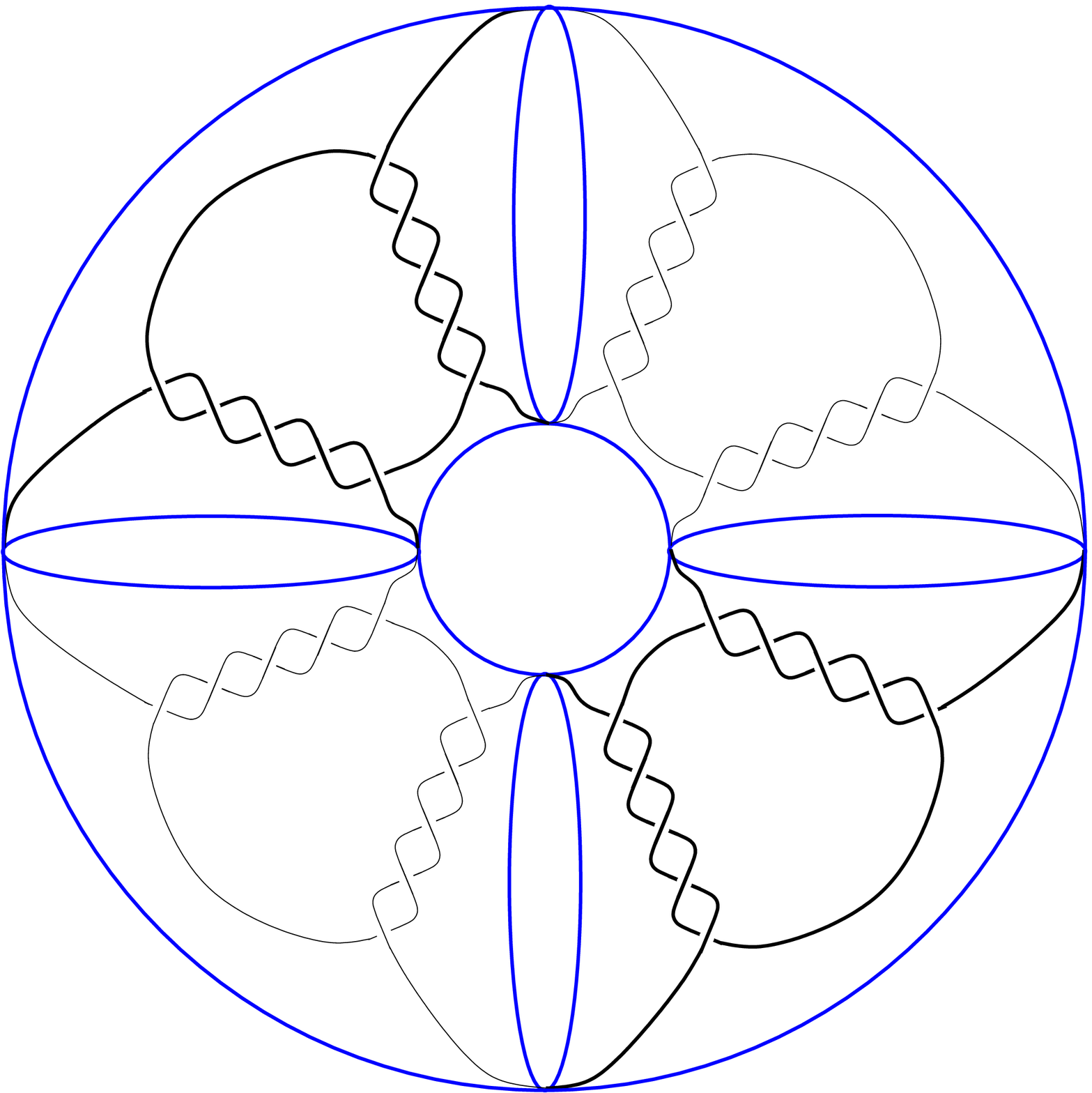}
  \end{center}
\end{minipage}
\begin{minipage}{2.4in}
  Alternating links are quasi-alternating, so we only consider
  non-alternating pretzel links. For all $p_i,\, q_j \geq 1$, let
  $P(p_1,\ldots,p_n, -q_1,\ldots, -q_m)$ denote the $(m+n)$--strand
  pretzel link.  As shown in figure on the left, the standard diagram
  of a non-alternating pretzel link can be made alternating on the
  torus and so has Turaev genus one.  
\end{minipage}
\\


\begin{prop}\label{pretzelKH}
  If $m,n \geq 2$ and all $p_i,\, q_j \geq 2$, then the Khovanov
  homology of $P(p_1,\ldots,p_n,-q_1,\ldots -q_m)$ has thickness exactly $3$. 
  More generally, any non-alternating Montesinos link, obtained 
by replacing $p_i>1$ (resp. $q_j>1$) half-twists with
a rational tangle that extends at least two of these crossings, also 
has thickness exactly $3$.
\end{prop}
\begin{proof} It is easy to check that for $n,\, m \geq 2$ and $p_i,\, q_j \geq
2$, the standard diagram for $P(p_1,\ldots,p_n,-q_1,\ldots,-q_m)$ is
adequate.  By Proposition 5.1 of \cite{KhPatterns}, it has thick Khovanov
homology.  Since the Turaev genus of non-alternating pretzel links is
one, the thickness bounds for Khovanov homology given in \cite{dkh,
  Manturov} are achieved by pretzel links whose Khovanov homology is
thick.  Since adequacy (and Turaev genus) is preserved after replacing
any half-twists with a rational tangle that extends at least two of these crossings,
the same argument applies to these Montesinos links.
\end{proof}
Similarly, it follows that every non-alternating Montesinos link (see
\cite{BZ} for their classification) has Turaev genus one.  However, some
of these have thin Khovanov homology.


\begin{theorem}\label{pretzeltheorem} \ 
\begin{enumerate}
\item For $n\geq 1$, $p_i \geq 1$ for all $i$, and  $q > \min(p_1,\ldots,p_n)$,
the pretzel link $P(p_1,\ldots,p_n, -q)$ is quasi-alternating. 
\item For $m,n \geq 2$ and all $p_i,\, q_j \geq 2$, the pretzel link $P(p_1,\ldots,p_n,-q_1,\ldots -q_m)$ is not quasi-alternating.
\end{enumerate}
Both statements are true for all permutations of $p_i$'s and $q_j$'s, and reflections of all these pretzel links. 
\end{theorem}
\begin{proof} 
We first prove part (1) by induction on $n$.  For
$n=1$, $P(p_1,-q) = T(2,p_1-q)$, which is quasi-alternating since
$p_1-q\neq 0$.

Let $n\geq 1$ and $p_i \geq 1$ for all $i$.  Suppose for $q >
\min(p_1,\ldots,p_n),\ P(p_1,\ldots,p_n, -q)$ is quasi-alternating.
Consider $P(p_1,\ldots,p_n,p_{n+1}, -q)$.  Without loss of generality,
$q > p_{i_0}$ for some $i_0$, with $1\leq i_0 \leq n$.
%
Let $L=P(p_1,\ldots,p_n,1,-q)$, and let $c$ be the crossing on the
$(n+1)$st strand.  Let $T(2,k)$ denote the alternating torus link.
Then $L_0=T(2,p_1)\# \ldots \# T(2,p_n)\#T(2,-q)$
is quasi-alternating by Lemma \ref{connectsum}, and
$L_{\infty}=P(p_1,\ldots,p_n,-q)$ is quasi-alternating by the
induction hypothesis.

By Lemma 4.3 of \cite{DFKLS2}, 
$${\rm det}(P(p_1,\ldots, p_n,-q_1,\ldots,-q_m)) = 
\left|\prod_{i=1}^n p_i \prod_{j=1}^m q_j \left(\sum_{i=1}^n \frac{1}{p_i} - 
\sum_{j=1}^m \frac{1}{q_j}\right)\right|.$$
Now, $q>p_{i_0}$ implies $\displaystyle{\frac{1}{p_{i_0}}-\frac{1}{q}>0}$, so 
$$q \prod_{i=1}^n p_i \left(1+\sum_{i\neq i_0}\frac{1}{p_i}+\left(\frac{1}{p_{i_0}}-\frac{1}{q}\right)\right)
=q \prod_{i=1}^n p_i \left(\sum_{i\neq i_0}\frac{1}{p_i}+\left(\frac{1}{p_{i_0}}-\frac{1}{q}\right)\right)
+ q\prod_{i=1}^n p_i.$$
Therefore, $\det(L)=\det(L_0)+\det(L_{\infty})$, which proves that
$L=P(p_1,\ldots,p_n,1,-q)$ is quasi-alternating at $c$. By Theorem
\ref{maintheorem}, $L=P(p_1,\ldots,p_n,p_{n+1},-q)$ is quasi-alternating
for all $p_{n+1} \geq 1$.  This completes the proof of part (1) by induction.

By Proposition \ref{pretzelKH} the pretzel links in part (2) have thick Khovanov homology, so they are not quasi-alternating.
%

The arguments above remain essentially unchanged for all permutations of $p_i$'s
and $q_j$'s, and for all reflections, given by negating every $p_i$
and $q_j$.  
\end{proof}

\begin{remark}\rm
  Widmer \cite{Widmer} extended Theorem \ref{pretzeltheorem} to certain Montesinos links.
\end{remark} 

\subsection{$3$--strand pretzel links}\label{3strands}
Theorem \ref{pretzeltheorem} still leaves open the quasi-alternating
status of many pretzel links.  For certain $3$--strand pretzel links,
this can be deduced from previous results.  As above, the statements
below for $P(p_1,p_2,-q)$ are true for all permutations of $p_i$'s and
$q$, and reflections of all these pretzel links.

According to \cite{Josh_Greene, Lisca_Stipsicz}, for the pretzel link $L=P(p_1,p_2,-q)$
with $p_1, p_2, q\geq 2$, $\Sigma(L)$ is an {\rm L}--space if and only if
\begin{enumerate}
\item $q \geq \min\{p_1,p_2\}$; or
\item $q = \min\{p_1,p_2\}-1\ $ and $\ \max\{p_1,p_2\}\leq 2q+1$.
\end{enumerate}
By Proposition \ref{Lprop}, $3$--strand pretzel links that satisfy
neither (1) nor (2) above are not quasi-alternating.  
%
%
Together with Theorem \ref{pretzeltheorem}(1), this leaves open
the quasi-alternating status for the following $3$--strand pretzel links:
%
%
 \begin{question}\
   \begin{enumerate}
   \item For $p \geq q \geq 2$, is $P(p,q,-q)$ quasi-alternating?
 \item For $ 2q +1 \geq  p \geq q+1 \geq 3$, is $P(p, q+1, -q)$ quasi-alternating?
   \end{enumerate}
   \end{question}
   Note that $P(3,3,-3)= 9_{46}$ and $P(4,3,-3)=10_{140}$.  In
   addition, thick Khovanov homology or knot Floer homology precludes
   some of these links from being quasi-alternating (see \cite{OSz_2004,
     Eftekhary, Suzuki}).  For example, $P(k,2,-2)$ has thick Khovanov
   homology for $2 \leq k \leq 5$ according to KhoHo \cite{KhoHo}.

\section{Quasi-alternating knots up to 10 crossings}
\label{10crossqa}

Manolescu \cite{Manolescu08} showed that all KH-thin knots up to 9 crossings,
except $9_{46}$ are quasi-alternating. Among 42 non-alternating
10-crossing knots, 32 are KH-thin.  Baldwin \cite{Baldwin08} showed
that 10 knots among these, which are closed $3$--braids, are
quasi-alternating.  Greene \cite{Josh_Greene} showed that eight more
knots $10_{150}$,$10_{151}$,$10_{156}$, $10_{158}$, $10_{163}$,
$10_{164}$, $10_{165}$, $10_{166}$ are quasi-alternating.  We show
that except for $10_{140}$, the remaining 13 knots are
quasi-alternating.  In the table below, we give the knot and its
Conway notation (see \cite{Kawauchi}).  For our computations, we replaced
the rational tangle indicated in bold with a crossing of same sign,
and checked that this crossing is quasi-alternating in the new
diagram. It follows from Theorem \ref{maintheorem} that the knot is
quasi-alternating.

\begin{center}
\begin{tabular}{|c|c||c|c||c|c|}
\hline
$10_{129}$ & $\textbf{32}, 21, 2-$ & $10_{135}$ & $221, \textbf{21}, 2-$ &
$10_{146}$ & $22, \textbf{21}, 21-$ \\
\hline 
$10_{130}$ & $311, \textbf{3}, 2-$ & $10_{137}$ & $\textbf{22}, 211, 2-$ &
$10_{147}$ & $211, \textbf{3}, 21-$ \\
\hline
$10_{131}$ & $311, \textbf{21}, 2-$ & $10_{138}$ & $\textbf{211}, 211, 2-$ &
$10_{160}$ & $-30: \textbf{20}: 20$ \\
\hline
$10_{133}$ & $\textbf{23}, 21, 2-$ & $10_{142}$ & $31, \textbf{3}, 3-$ & & \\
\hline
$10_{134}$ & $221, \textbf{3}, 2-$ & $10_{144}$ & $31, \textbf{21}, 21-$ & & \\
\hline
\end{tabular}
\end{center}

As a result, all KH-thin knots up to 10 crossings, except $9_{46}$ and
$10_{140}$, are quasi-alternating.  Shumakovitch \cite{Shumakovitch}
informed us that $9_{46}$ and $10_{140}$ have thick odd Khovanov
homology, so they are not quasi-alternating.

\subsection*{Acknowledgements}
We thank Peter Ozsv\'ath for useful discussions, which included the
proof of Lemma \ref{connectsum}.  We thank Josh Greene for very helpful discussions
about results in Sections \ref{3strands} and \ref{10crossqa}.  We
thank Alex Shumakovitch for informing us about $9_{46}$ and
$10_{140}$.  We thank the anonymous referee for thoughtful comments.

\bibliography{qatwist}

\providecommand{\bysame}{\leavevmode\hbox to3em{\hrulefill}\thinspace}
\providecommand{\MR}{\relax\ifhmode\unskip\space\fi MR }
\providecommand{\MRhref}[2]{%
  \href{http://www.ams.org/mathscinet-getitem?mr=#1}{#2}
}
\providecommand{\href}[2]{#2}
\begin{thebibliography}{10}

\bibitem{Baldwin08}
J.~Baldwin, \emph{{Heegaard Floer homology and genus one, one boundary
  component open books}},  (arXiv:0804.3624v2 [math.GT]).

\bibitem{BZ}
G.~Burde and H.~Zieschang, \emph{Knots}, second ed., de Gruyter Studies in
  Mathematics, vol.~5, Walter de Gruyter \& Co., Berlin, 2003.

\bibitem{cjp}
A.~Champanerkar and I.~Kofman, \emph{On links with cyclotomic {J}ones
  polynomials}, Algebr. Geom. Topol. \textbf{6} (2006), 1655--1668.

\bibitem{KHshort}
\bysame, \emph{Spanning trees and {K}hovanov homology},  (to appear in {\em
  Proc. Amer. Math. Soc.}, arXiv:math/0607510v3 [math.GT]).

\bibitem{dkh}
A.~Champanerkar, I.~Kofman, and N.~Stoltzfus, \emph{Graphs on surfaces and
  {K}hovanov homology}, Algebr. Geom. Topol. \textbf{7} (2007), 1531--1540.

\bibitem{DFKLS2}
O.~T. Dasbach, D.~Futer, E.~Kalfagianni, X.~S. Lin, and N.~W. Stoltzfus,
  \emph{Alternating sum formulae for the determinant and other link
  invariants},  (arXiv:math/0611025v2 [math.GT]).

\bibitem{Eftekhary}
E.~Eftekhary, \emph{Heegaard {F}loer homologies of pretzel knots},
  (arXiv:math.GT/0311419).

\bibitem{Josh_Greene}
J.~Greene, \emph{{A spanning tree model for the Heegaard Floer homology of a
  branched double-cover}},  (arXiv:0805.1381v1 [math.GT]).

\bibitem{Hironaka}
E.~Hironaka, \emph{The {L}ehmer polynomial and pretzel links}, Canad. Math.
  Bull. \textbf{44} (2001), no.~4, 440--451.

\bibitem{Kawauchi}
A.~Kawauchi, \emph{A survey of knot theory}, Birkh\"auser Verlag, Basel, 1996,
  Translated and revised from the 1990 Japanese original by the author.
  \MR{MR1417494 (97k:57011)}

\bibitem{KhPatterns}
M.~Khovanov, \emph{Patterns in knot cohomology. {I}}, Experiment. Math.
  \textbf{12} (2003), no.~3, 365--374.

\bibitem{Lisca_Stipsicz}
P.~Lisca and A.~Stipsicz, \emph{{Ozsvath-Szabo invariants and tight contact
  three-manifolds, III}},  (math.SG/0505493).

\bibitem{Lowrance}
A.~Lowrance, \emph{On knot {F}loer width and {T}uraev genus},
  (arXiv:0709.0720v1).

\bibitem{Manolescu08}
C.~Manolescu, \emph{An unoriented skein exact triangle for knot {F}loer
  homology}, Math. Res. Lett. \textbf{14} (2007), no.~5, 839--852.

\bibitem{CM_PO}
C.~Manolescu and P.~Ozsv{\'a}th, \emph{{On the Khovanov and knot Floer
  homologies of quasi-alternating links}},  (arXiv:0708.3249v2 [math.GT]).

\bibitem{Manturov}
V.~Manturov, \emph{Minimal diagrams of classical and virtual links},
  (arXiv:math.GT/0501393).

\bibitem{OSz_2004}
P.~Ozsv{\'a}th and Z.~Szab{\'o}, \emph{Knot {F}loer homology, genus bounds, and
  mutation}, Topology Appl. \textbf{141} (2004), no.~1-3, 59--85.

\bibitem{OzSz_double_covers}
\bysame, \emph{On the {H}eegaard {F}loer homology of branched double-covers},
  Adv. Math. \textbf{194} (2005), no.~1, 1--33.

\bibitem{KhoHo}
A.~Shumakovitch, \emph{{KhoHo}}, {\rm available from
  http://www.geometrie.ch/KhoHo/} (2003).

\bibitem{Shumakovitch}
A.~Shumakovitch, \emph{Private communication}, March 2008.

\bibitem{Suzuki}
R.~Suzuki, \emph{{Khovanov homology and Rasmussen's s-invariants for pretzel
  knots}},  (arXiv:math.QA/0610913).

\bibitem{thistlethwaite}
M.~Thistlethwaite, \emph{A spanning tree expansion of the {J}ones polynomial},
  Topology \textbf{26} (1987), 297--309.

\bibitem{Widmer}
T.~Widmer, \emph{{Quasi-alternating Montesinos links}},  (arXiv:math/0811.0270
  [math.GT]).

\end{thebibliography}
\bibliographystyle{amsplain}
\end{document}